\documentclass{amsart}
\usepackage[T1]{fontenc}
\usepackage{amssymb}
\usepackage{amsmath}
\usepackage{amsfonts}
\usepackage{geometry}
\usepackage{hyperref}

\hypersetup{
    colorlinks=true,
    linkcolor=black,
    citecolor=black,
    filecolor=black,
    urlcolor=black,
}

\newtheorem{theorem}{Theorem}
\theoremstyle{plain}

\newtheorem{definition}[theorem]{Definition}
\newtheorem{example}[theorem]{Example}

\newtheorem{lemma}[theorem]{Lemma}

\newtheorem{proposition}[theorem]{Proposition}
\newtheorem{remark}[theorem]{Remark}

\numberwithin{equation}{section}
\input{tcilatex}

\renewcommand{\P}{ {\mathbb{P}} }

\begin{document}
\title{Stochastic control with rough paths}
\author{Joscha Diehl, Peter Friz and Paul Gassiat}

\begin{abstract}
We study a class of controlled rough differential equations. It is shown
that the value function satisfies a HJB type equation; we also establish a
form of the Pontryagin maximum principle. Deterministic problems of this
type arise in the duality theory for controlled diffusion processes and
typically involve anticipating stochastic analysis. We propose a formulation
based on rough paths and then obtain a generalization of Roger's duality
formula [L. C. G. Rogers, Pathwise Stochastic Optimal Control. SIAM J.
Control Optim. 46, 3, 1116-1132, 2007] from discrete to continuous time. We
also make the link to old work of [M. H. A. Davis and G. Burstein, A
deterministic approach to stochastic optimal control with application to
anticipative optimal control. Stochastics and Stochastics Reports,
40:203--256, 1992].
\end{abstract}

\subjclass{Primary 60H99}
\keywords{Stochastic control, Duality, Rough Paths}
\thanks{ The first author is supported by the DFG project SPP1324. The last
two authors have received partial funding from the European Research Council
under the European Union's Seventh Framework Programme (FP7/2007-2013) / ERC
grant agreement nr. 258237. We thank the attendants of the SSSC 2012
workshop, Madrid, for valuable discussions. }
\maketitle



\section{Introduction}


In classical works \cite{MR0451404,MR0461664} Doss and Sussmann studied the
link between ordinary and stochastic differential equations (ODEs and SDEs,
in the sequel). In the simplest setting, consider a nice vector field $%
\sigma $ and a \textit{smooth} path $B:\left[ 0,T\right] \rightarrow \mathbb{%
R}$, one solves the (random) ordinary differential equation%
\begin{equation*}
\dot{X}=\sigma \left( X\right) \dot{B},
\end{equation*}
so that $X_{t}=e^{\sigma B_{t}}X_{0}$, where $e^{\sigma B_{t}}$ denotes
flow, for unit time, along the vector field $\sigma \left( \cdot \right)
B_{t}$. The point is that the resulting formula for $X_{t}$ makes sense for
any continuous path $B$, and in fact the It\^{o}-map $B\mapsto X$ is
continuous with respect to $\left\Vert \cdot \right\Vert _{\infty ;\left[ 0,T%
\right] }$. In particular, one can use this procedure for every
(continuous)\ Brownian path; the so-constructed SDE solution then solves the
Stratonovich equation,%
\begin{equation*}
dX=\sigma \left( X\right) \circ dB=\sigma \left( X\right) dB+\frac{1}{2}%
\left( \sigma \sigma ^{\prime }\right) \left( X\right) dt.
\end{equation*}%
When $B=B\left( \omega \right) :\left[ 0,T\right] \rightarrow \mathbb{R}^{d}$
is a multidimensional Brownian motion, which we shall assume from here on,
this construction fails and indeed the It\^{o}-map is notorious for its lack
of (classical) continuity properties. Nonetheless, many approximations, $%
B^{n}\rightarrow $ $B$, examples of which include the piecewise linear -,
mollifier - and Karhunen-Love approximations, have the property that the
corresponding (random) ODE\ solutions, say to $dX^{n}=b\left( X^{n}\right)
dt+\sigma \left( X^{n}\right) dB^{n}$ with $\sigma =\left( \sigma _{1},\dots
,\sigma _{d}\right) $, convergence to the solution of the Stratonovich
equation%
\begin{equation*}
dX=b\left( X\right) dt+\sigma \left( X\right) \circ dB.
\end{equation*}%
(At least in the case of piecewise linear approximations, this result is
known as Wong-Zakai theorem\footnote{%
... although, strictly speaking, the multi-dimensional case is due to M.
Clark and Stroock--Varadhan.}) It was a major step forward, due to T. Lyons 
\cite{lyons98}, to realize that the multidimensional SDE case can also be
understood via deterministic differential equations (known as \textit{rough
differential equations}); they do, however, require more complicated driving
signals (known as \textit{rough paths}) which in the present context are of
the form%
\begin{equation*}
\mathbf{B}\left( \omega \right) :\left[ 0,T\right] \rightarrow \mathbb{R}%
^{d}\oplus so\left( d\right) ,
\end{equation*}%
and "contain", in addition to the Brownian sample path, \textit{L\'{e}vy
stochastic area}, viewed as process with values in $so\left( d\right) $.
Among the \textit{many} applications of rough paths to stochastic analysis
(eg. \cite{FV} and the references therein) let us just mention here that (i) 
\textit{all Wong--Zakai type results} follow from $B^{n}\rightarrow $ $%
\mathbf{B}$ in a rough path metric (ii) the (rough)pathwise resolution of
SDE can handle immediately situations with anticipating randomness in the
coefficients; consistency with anticipating SDE in the sense of
Ocone--Pardoux \cite{MR0995291} was established in \cite{CFV}.

The purpose of this paper is to explore the \textit{interplay of rough paths
with control problems}; more specifically, in the context of \textit{%
controlled} differential equations with usual aim of maximizing a certain
payoff function over a fixed time horizon $\left[ 0,T\right] $. As is well
known, the dynamic programming principle, in the context of ODEs, leads to
Hamilton--Jacobi (HJ) equations for the value function, i.e. non-linear
first order partial differential equations, for the value function. Optimal
control can be characterized with the aid of the Pontryagin maximum
principle (PMP) which essentially is the method of characteristics applied
to the HJ\ equations. As will be discussed in detail in \textbf{Section \ref%
{sec:deterministicControl}} of this paper, all this can be done for general
(i.e. deterministic) rough differential equations, say of the form%
\begin{equation*}
dX=b\left( X,\mu \right) dt+\sigma \left( X\right) d\mathbf{\eta ,}
\end{equation*}%
where $\mathbf{\eta }$ is a rough path and $\mu =\left( \mu _{t}\right) $ a
control. The value function then solves a so-called rough viscosity
equation; as introduced in \cite{CFO,D12}; see also the pathwise stochastic
control problems proposed by Lions--Souganidis \cite{LS98b}, further studied
by Buckdahn--Ma \cite{Bu}. One can, of course, apply this (rough)pathwise to
SDEs, just take $\mathbf{\eta =B}\left( \omega \right) $; however, the
(optimal) control $\mu =\mu _{t}\left( \omega \right) $ will depend
anticipatingly on the Brownian driver. Moreover, the $\omega $-wise
optimization has (at first glance) little to with the classical stochastic
control problem in which one maximizes the expected value, i.e. an average
over all $\omega $'s, of a payoff function.

Making the link between deterministic and such classical stochastic control
problem is the purpose of \textbf{Section \ref{sec:dualityResults}}. In a
discrete time setting, Wets \cite{Wets} first observed that stochastic
optimization problems resemble deterministic optimization problems up to the
nonanticipativity restriction on the choice of the control policy. The
continuous time setting, i.e. studying the link between \textit{controlled}
ordinary and \textit{controlled} stochastic differential equations, goes
back to Davis--Burstein \cite{DB}. The (meta) theorem here is actually a
duality of the form%
\begin{equation*}
\left( D1\right) :\sup_{\nu }\mathbb{E}\left[ ...\right] =\mathbb{E}\left[
\sup_{\mu }\left[ ...+P^{\ast }\right] \right]
\end{equation*}%
for a suitable penality $P^{\ast }$, or more generally%
\begin{equation*}
\left( D2\right) :\sup_{\nu }\mathbb{E}\left[ ...\right] =\inf_{P}\mathbb{E}%
\left[ \sup_{\mu }\left[ ...+P\right] \right] .
\end{equation*}%
where $\nu $ denotes an adapted control, $\mu $ a possibly anticipating
control, the dots $\dots $ may stand for a payoff such as $g\left(
X_{T}\right) $. Note that $\left( D2\right) $ has an immediate practical
advantage: any choice of $P$ gives an upper bound ("duality bound") on the
value function and therefore complements lower bounds obtained from picking
a particular strategy $\nu $.

In the context of continuous time optimal stopping problems, $\left(
D1\right) $ was established by Davis--Karatzas \cite{DK}, $\left( D2\right) $
is due to Rogers \cite{R02}, see also \cite{HL}, with important applications
to the (Monte Carlo) pricing of American options. The extension of $\left(
D2\right) $, with concrete martingale penality terms, to general control
problem in \textit{discrete time only} was carried out by Rogers in \cite%
{R07}, see also \cite{brownSmithSun}.

While the discrete time setting of Rogers apriori avoids all technicalities
related to measurability (e.g. $\sup_{\mu }$ vs. $ess\sup_{\mu }$, the
meaning of $X$ controlled by anticipating $\mu ^{\ast }$ ...)\ such things
obviously matter in the work of Davis--Burstein. And indeed, the
Ocone--Pardoux techniques (notably the representation of anticipating SDE
solutions via flow decomposition) play a key role in their work.

Our contribution in \textbf{section \ref{sec:dualityResults}} is then
twofold:

(I). We give a "general" duality result in continuous time and see how it
can be specialized to yield a version of Roger's duality $\left( D2\right) $
for control problems in continous time. Another specialization leads to the
Davis--Burstein result, which we review in this context. (Both approaches
are then compared explicitly via computations in LQC problems.)

(II) We make the case that \textit{rough path analysis }is ideal to
formulate and analyze these problems. Indeed, it allows to write down, in a
meaningful and direct way, all the quantities that one wants to write down -
without any headache related to afore-mentioned
(measurability/anticipativity) technicalities : throughout, all quantities
depend continuously on some abstract rough path $\mathbf{\eta }$ and are
then \ - trivially - measurable upon substituion\textbf{\ }$\mathbf{\eta
\leftarrow B}\left( \omega \right) $.

\section{Notation}

For $\alpha > \frac{1}{3}$ denote by $\mathcal{C}^{0,\alpha} = \mathcal{C}%
^{0,\alpha}(E)$ the space of geometric, $\alpha$-H\"{o}lder rough paths over 
$E$, where $E$ is a Banach space chosen according to context. \footnote{%
For $\alpha > \frac{1}{2}$ these are just H\"older continuous $E$-valued
paths. For $\alpha \in (1/3,1/2]$ additional ``area'' information is
necessary. We refer to \cite{LQ} and \cite{FV} for background on rough path
theory.} %
%
On this space, we denote by $\rho_{\alpha-\text{H\"{o}l}}$ 
the corresponding inhomgenous distances.

Let $U$ be some separable metric space (the control space). Denote by $%
\mathcal{M}$ the class of measurable controls $\mu: [0,T] \to U$.

When working on a filtered probability space $(\Omega, {\mathcal{F}} ,\P )$, 
$\mathcal{A}$ will denote the class of progressively measurable controls $%
\nu: \Omega \times [0,T] \to U$.

\section{Deterministic control with rough paths}

\label{sec:deterministicControl}

Let $\eta :[0,T]\rightarrow \mathbb{R}^{d}$ be a smooth path. Write $%
X=X^{t,x,\mu}$ for the solution to the controlled ordinary differential
equation%
\begin{equation}
dX^{t,x,\mu}_s = b\left( X^{t,x,\mu}_s,\mu_s\right) ds+\sigma \left(
X^{t,x,\mu}_s \right) d\eta_s, \qquad X^{t,x,\mu}_t = x  \label{dX_ODE}
\end{equation}%
Classical control theory allows to maximize $\int_{t}^{T}f\left(
s,X_{s}^{t,x,u},\mu_{s}\right) ds+g\left( X_{T}^{\mu}\right) $ over a class
of admissible controls $\mu$. As is well-known, 
\begin{equation}
v\left( t,x\right) :=\sup_{\mu}\left\{ \int_{t}^{T}f\left(
s,X_{s}^{t,x,\mu},\mu_{s}\right) ds+g\left( X_{T}^{t,x,\mu}\right) \right\}
\label{V_ODE}
\end{equation}%
is the (under some technical conditions:\ unique, bounded uniformly
continuous) viscosity solution to the HJB\ equation%
\begin{align*}
-\partial _{t}v - H\left( t,x,Dv\right) - \left\langle \sigma \left(
x\right), Dv\right\rangle \dot{\eta} &= 0, \\
v(T,x) &= g(x).
\end{align*}
where $H$ acting on $v$ is given by 
\begin{equation}  \label{equ:H}
H\left( t, x, p\right) = \sup_{u}\left\{ \left\langle b\left( x, u\right) ,
p\right\rangle +f\left( t,x,u\right) \right\}
\end{equation}
Now, \eqref{dX_ODE} also makes sense for a driving rough path (including a
controlled drift terms to the standard setting of RDEs is fairly
straight-forward; for the reader's convenience proofs are given in the
appendix). This allows to consider the optimization problem \eqref{V_ODE}
for controlled RDEs. 

\subsection{HJB equation}

The main result here is that the corresponding value function satisfies a
``rough PDE''. Such equations go back to Lions-Souganidis (in \cite{LS98b}
they consider a pathwise stochastic control problem and give an associated
stochastic HJB equation, see also \cite{Bu}; these correspond to ${\mathbf{%
\eta}}=\mathbf{B}(\omega)$ in the present section). However their
(non-rough) pathwise setup is restricted to commuting diffusion vector
fields $\sigma_1, \ldots, \sigma_d$ (actually, \cite{LS98b} consider
constant vector fields). Extensions to more general vector fields via a
rough pathwise approach was then obtained in \cite{CFO} (see also \cite{DFO}%
).

\begin{definition}
\label{defRoughPDE} Let $\mathbf{\eta } \in \mathcal{C}^{0,\alpha}$ be a
geometric rough path, $\alpha \in (0,1]$. Assume $F,G,\phi$ to be such that
for every smooth path $\eta$ there exists a unique BUC viscosity solution to 
\begin{align*}
-\partial_t v^\eta - F(t,x,v^\eta,Dv^\eta,D^2v^\eta) - G(t,x,v^\eta,Dv^\eta) 
\dot{\eta}_t &= 0, \\
v^\eta(T,x) &= \phi(x).
\end{align*}

We say that $v \in BUC$ solves the \textit{rough partial differential
equation} 
\begin{align*}
- dv - F(t,x,v,Dv,D^2v) dt - G(t,x,v,Dv) d\mathbf{\eta}_t &= 0, \\
v(T,x) &= \phi(x),
\end{align*}
if for every sequence of smooth paths $\eta^n$ such that $\eta^n \to \mathbf{%
\eta}$ in rough path metric we have locally uniformly 
\begin{align*}
v^{\eta^n} \to v.
\end{align*}
\end{definition}

\begin{remark}
\begin{enumerate}
\item We remark that uniqueness of a solution, if it exists, is built into
the definition (by demanding uniqueness for the approximating problems).

\item In special cases (in particular the gradient noise case of the
following theorem) it is possible to define the solution to a rough PDE
through a coordinate transformation, if the vector fields in front of the
rough path are smooth enough. This approach is followed in \cite{LSSemi}.

The two definitions are equivalent, if the coefficients admit enough
regularity (see \cite{CFO}). In the following theorem the coordinate
transformation is not applicable, since $\sigma$ is only assumed to be $%
\func{Lip}^\gamma$ instead of $\func{Lip}^{\gamma+2}$.
\end{enumerate}
\end{remark}

\begin{theorem}
\label{thm:roughHJ} Let $\mathbf{\eta } \in \mathcal{C}^{0,\alpha}( {\mathbb{%
R}} ^d)$ be a rough path, $\alpha \in (\frac{1}{3},\frac{1}{2}]$. Let $%
\gamma > 1/\alpha$. Let $b: {\mathbb{R}} ^e \times U \to {\mathbb{R}} ^e$ be
continuous and let $b(\cdot,u) \in \func{Lip}^1( {\mathbb{R}} ^e)$ uniformly
in $u \in U$. Let $\sigma_1,\dots,\sigma_d \in \func{Lip}^\gamma( {\mathbb{R}%
} ^e)$. Let $g \in BUC( {\mathbb{R}} ^e)$. Let $f: [0,T] \times {\mathbb{R}}
^e \times U \to {\mathbb{R}} $ be bounded, continuous and locally uniformly
continuous in $t,x$, uniformly in $u$.

For $\mu \in \mathcal{M}$ consider the RDE with controlled drift \footnote{%
Extension to time-dependent $b$, $\sigma$ would be straightforward, see \cite%
[ch. 12]{FV}. It is also possible to consider the controlled hybrid RDE/SDE $%
dX = b(X,\mu) dt + \tilde \sigma(X, \mu) dW + \sigma(X) d{\mathbf{\eta}}$,
see \cite{D12}.} (Theorem \ref{Thm:RDEWithControlledDrift}), 
\begin{equation}  \label{eq:controlledRDE}
dX^{t,x,\mu,\mathbf{\eta}}=b\left( X^{t,x,\mu,\mathbf{\eta}},\mu\right)
dt+\sigma \left( X^{t,x,\mu,\mathbf{\eta}}\right) d\mathbf{\eta}, \quad
X^{t,x,\mu,\mathbf{\eta}}_t = x.
\end{equation}
Then 
\begin{equation*}
v\left( t,x\right) := v^{\mathbf{\eta }}\left( t,x\right) := \sup_{\mu \in 
\mathcal{M}} \left\{ \int_{t}^{T}f\left( s,X_{s}^{t,x,\mu,{\mathbf{\eta}}%
},\mu_{s}\right) ds+g\left( X_{T}^{t,x,\mu,{\mathbf{\eta}}}\right) \right\}
\end{equation*}%
is the unique bounded, uniformly continuous viscosity solution to the rough
HJ equation%
\begin{align}  \label{equ:RoughHJB}
\begin{split}
-dv - H\left( x,Dv\right) dt-\left\langle \sigma \left( x\right)
,Dv\right\rangle d\mathbf{\eta } &= 0, \\
v(T,x) &= g(x).
\end{split}%
\end{align}
\end{theorem}

\begin{proof}
The case $f=0, \sigma \in \func{Lip}^{\gamma + 2}$ appears in \cite{D12}.
The general case presented here is different, since we cannot use a
coordinate transformation. Let a smooth sequence $\eta^n$ be given, such
that $\eta^n \to \mathbf{\eta}$ in $\mathcal{C}^{0,\alpha}$. Let 
\begin{align*}
v^n(t,x) := \sup_{\mu\in\mathcal{M}} \Xi_{t,x}[\eta^n,\mu],
\end{align*}
where $\Xi_{t,x}[\mathbf{\gamma},\mu] := \int_{t}^{T} f\left(
s,X_{s}^{t,x,\mu,\mathbf{\gamma}},\mu_{s}\right) ds + g\left( X_{T}^{t,x,\mu,%
\mathbf{\gamma}}\right)$ for any (rough) path $\mathbf{\gamma}$. By
classical control theory (e.g. Corollary III.3.6 in \cite{BCD}) 
we have that $v^n$ is the unique bounded, continuous viscosity solution to 
\begin{align*}
-dv^n - H\left( x,Dv^n\right) dt-\left\langle \sigma \left( x\right)
,Dv^n\right\rangle d\eta^n &= 0, \\
v^n(T,x) &= g(x).
\end{align*}

Then 
\begin{align*}
|v^n(t,x) - v(t,x)| &\le \sup_{\mu \in \mathcal{M}} \left| \Xi_{t,x}[{%
\mathbf{\eta}},\mu] - \Xi_{t,x}[\eta^n,\mu] \right|.
\end{align*}
Note that $\Xi$ is continuous in $\mathbf{\gamma}$ uniformly in $\mu$ (and $%
(t,x)$) by Theorem \ref{Thm:RDEWithControlledDrift}. Therefore, $v^n$
converges locally uniformly to $v$ 
and then, by Definition \ref{defRoughPDE}, $v$ solves \eqref{equ:RoughHJB}.
\end{proof}

\begin{example}
\label{ex:additive} In the case with additive noise ($\sigma(x) \equiv Id$)
and state-independent gains / drift ($f(s,x,u) = f(s,u), b(x,u)=b(s,u)$),
this rough deterministic control problem admits a simple solution. Indeed,
if $v^0$ is the value function to the standard deterministic problem for ${%
\mathbf{\eta}} \equiv 0$, i.e. 
\begin{align*}
v^0(t,x) &= \sup_{\mu \in \mathcal{M}} \left\{ \int_t^T f(s,\mu_s) ds +
g\left( x + \int_t^T b(s,\mu_s)ds\right) \right\},
\end{align*}
then one has immediately (since $\eta$ only appears in the terminal gain) 
\begin{align*}
v^{\mathbf{\eta}}(t,x) &= v^0(t, x + \eta_T - \eta_t).
\end{align*}
When $v^0$ has a nice form, this gives simple explicit solutions. For
instance, assuming in addition $f \equiv 0$, $U$ convex and $b(s,u)=u$, $v^0$
is reduced to a static optimization problem and 
\begin{align*}
v^{\mathbf{\eta}}(t,x) &= \sup_{u \in U} g\big(x + \eta_T - \eta_t + (T-t) u%
\big).
\end{align*}
\end{example}

\subsection{Pontryagin maximum principle}

If $\eta$ is smooth, then Theorem 3.2.1 in \cite{YZ} gives the following
optimality criterium.

\begin{theorem}
\label{thm:classicalPMP} Let $\eta$ be a smooth path. Assume $b,f,g$ be $C^1$
in $x$, such that the derivative is Lipschitz in $x,u$ and bounded and let $%
\sigma,g$ be $C^1$ with bounded, Lipschitz first derivative. Let $\bar X,
\bar \mu$ be an optimal pair for problem \eqref{V_ODE} with $t=0$. Let $p$
be the unique solution to the backward ODE 
\begin{align*}
-\dot p(t) &= Db(\bar X_t, \bar \mu_t) p(t) + D\sigma(\bar X_t) \dot{\eta}_t
p(t) + Df(\bar X_t, \bar \mu_t), \\
p(T) &= D g( \bar X_T ).
\end{align*}

Then 
\begin{align*}
b(\bar X_t, \bar \mu_t) p(t) + f(\bar X_t, \bar \mu_t) = \sup_{u \in U} 
\left[ b(\bar X_t, u) p(t) + f(\bar X_t, u) \right], \qquad a.e.\ t \in
[0,T].
\end{align*}
\end{theorem}

Let now $\mathbf{\eta}$ be rough. We have the following equivalent statement.

\begin{theorem}
Assume the same regularity on $b,f,g$ as in Theorem \ref{thm:classicalPMP}.
Assume $\sigma_1,\dots,\sigma_d \in \func{Lip}^{\gamma+2}( {\mathbb{R}} ^e)$%
. Let ${\mathbf{\eta}} \in \mathcal{C}^{0,\alpha}$ be a geometric rough
path, $\alpha \in (1/3,1/2]$. 
Let $\bar X, \bar \mu$ be an optimal pair. Let $p$ be the unique solution to
the controlled, backward RDE 
\begin{align*}
-d p(t) &= Db(\bar X_t, \bar \mu_t) p(t) dt + D\sigma(\bar X_t) p(t) d{%
\mathbf{\eta}}_t + Df(\bar X_t, \bar \mu_t) dt, \\
p(T) &= D g( \bar X_T ).
\end{align*}

Then 
\begin{align*}
b(\bar X_t, \bar \mu_t) p(t) + f(\bar X_t, \bar \mu_t) = \sup_{u\in U} \left[
b(\bar X_t, u) p(t) + f(\bar X_t, u) \right], \qquad a.e.\ t.
\end{align*}
\end{theorem}

\begin{remark}
This is the necessary condition for an admissible pair to be optimal. In the
classical setting there do also exist sufficient conditions (see for example
Theorem 3.2.5 in \cite{YZ}). They rely on convexity of the Hamiltonian and
therefore will in general not work in our setting because, informally, the $d%
{\mathbf{\eta}}$-term switches sign all the time.
\end{remark}

%
We prepare the proof with the following Lemma.

\begin{lemma}
\label{lem:spikeVariation} Let $\bar X, \bar \mu$ be an optimal pair. Let $%
\mu$ be any other control. Let $I \subset [0,T]$ be an interval with $|I| =
\varepsilon$. Define 
\begin{align*}
\mu^\varepsilon(t) := 1_I(t) \mu(t) + 1_{[0,T]\setminus I}(t) \bar \mu(t).
\end{align*}
Let $X^\varepsilon$ be the solution to the controlled RDE %
\eqref{eq:controlledRDE} corresponding to the control $\mu^\varepsilon$. Let 
$Y^\varepsilon$ be the solution to the RDE 
\begin{align*}
Y^\varepsilon_t = \int_0^t Db( \bar X_r, \bar \mu_r ) Y^\varepsilon_r dr +
\int_0^t D\sigma( \bar X_r ) Y^\varepsilon_r d{\mathbf{\eta}}_r + \int_0^t %
\left[ b( \bar X_r, \mu_r ) - b( \bar X_r \bar \mu_r ) \right] 1_I(r) dr.
\end{align*}

Then 

\begin{align}
    \label{equ:O1}
    \sup_t |X^\varepsilon_t - \bar X_t|
    &= O(\varepsilon), \\
    %
    %
    \label{equ:o1}
    \sup_t |X^\varepsilon_t - \bar X_t - Y^\varepsilon_t| &= o(\varepsilon), \\
    \label{equ:o2}
    \begin{split}
    J(\mu^\varepsilon) - J(\bar \mu)
    &=
    \langle Dg( \bar X_T ), Y^\varepsilon_T \rangle \\
    &\qquad
    +
    \int_0^T
    \left[
      \langle Df( \bar X_r, \bar \mu_r ), Y^\varepsilon_r \rangle
      +
      \left\{
        f( \bar X_r, \mu_r ) - f( \bar X_r, \bar \mu_r )
      \right\}
      1_I(r)
    \right] dr
    + o(\varepsilon).
    \end{split}
  \end{align}
  
\end{lemma}

\begin{proof}
The joint RDE reads as 
\begin{align*}
\bar X_t &= x + \int_0^t b\left( \bar X_r,\bar \mu_r \right) dr + \sigma
\left( \bar X_r \right) d\mathbf{\eta}_r \\
X^\varepsilon_t &= x + \int_0^t b\left( X^\varepsilon_r,\mu^\varepsilon_r
\right) dr + \sigma \left( X^\varepsilon_r \right) d\mathbf{\eta}_r \\
Y^\varepsilon_t &= \int_0^t Db( \bar X_r, \bar \mu_r ) Y^\varepsilon_r dr +
\int_0^t D\sigma( \bar X_r ) Y^\varepsilon_r d{\mathbf{\eta}}_r + \int_0^t %
\left[ b( \bar X_r, \mu_r ) - b( \bar X_r \bar \mu_r ) \right] 1_I(r) dr.
\end{align*}

By Theorem \ref{Thm:RDEWithControlledDrift} we can write the solution $%
S^\varepsilon = (\bar X, X^\varepsilon, Y^\varepsilon)$ as $\phi(t, \tilde
S^\varepsilon_t)$ where 
\begin{align*}
d\phi(t,s) = \hat\sigma( \phi(t,s) ) d{\mathbf{\eta}}_t
\end{align*}
and $\tilde S = (\tilde{ \bar{ X } }, \tilde X^\varepsilon, \tilde
Y^\varepsilon)$ with 
\begin{align*}
d\tilde{S}^\varepsilon_t = B(t, \tilde{S}^\varepsilon_t, (\bar \mu_t,
\mu^\varepsilon_t )) dt,
\end{align*}
where 
\begin{align*}
\hat\sigma( x_1, x_2, x_3 ) = \left( 
\begin{matrix}
\sigma(x_1) \\ 
\sigma(x_2) \\ 
D\sigma(x_1) x_3%
\end{matrix}
\right)
\end{align*}
and $B(t, (x_1,x_2,x_3), (u_1, u_2))$ is given by 
\begin{align*}
\left( 
\begin{matrix}
\partial_{x_1} \phi^{-1}_1(t, \phi(t,x)) b(\phi_1(t,x), u_1) \\ 
\partial_{x_2} \phi^{-1}_2(t, \phi(t,x)) b(\phi_2(t,x), u_2) \\ 
\partial_{x_3} \phi^{-1}_3(t, \phi(t,x)) \left[ Db(\phi_1(t,x), u_1)
\phi_3(t,x) + \left\{ b(\phi_1(t,x), u) - b(\phi_1(t,x),\bar u) \right\}
1_I(t) \right] \\ 
\qquad + \partial_{x_1} \phi^{-1}_3(t, \phi(t,x)) b(\phi_1(t,x), u_1)%
\end{matrix}
\right).
\end{align*}
From Lemma 3.2.2 in \cite{YZ} it follows that 
\begin{align*}
\sup_t |\tilde X^\varepsilon_t - \tilde{ \bar{ X_t } }| = O(\varepsilon).
\end{align*}
From this we can deduce \eqref{equ:O1} 
using the Lispchitzness of $\phi$. 
Furthermore, \eqref{equ:o1} follows from 
\begin{align*}
\sup_t |\tilde X^\varepsilon_t - \tilde{ \bar{ X_t } } - \tilde
Y^\varepsilon_t| = o(\varepsilon)
\end{align*}
and 
\begin{align*}
\phi_1(t, (a,b,b-a)) - \phi_2(t, (a,b,b-a)) - \phi_3(t, (a,b,b-a)) = O(
|b-a|^2 ).
\end{align*}
(note that $\partial_{x_1} \phi^{-1}_3(t, \phi(t, (x_1,x_2,0))) = 0$).
Finally \eqref{equ:o2} follows by direct calculation from \eqref{equ:O1} 
and \eqref{equ:o1}.
\end{proof}

\begin{proof}
We follow the idea of the proof of Theorem 3.2.1 in \cite{YZ}. Fix $x_0 \in {%
\mathbb{R}} ^e$. Without loss of generality we take $t=0$. 
Define for $\mu \in \mathcal{M}$ 
\begin{align*}
J( \mu ) := \int_{0}^T f(r,X^{0,x_0,\mu,{\mathbf{\eta}}}_r, \mu_r) dr + g(
X^{0,x_0,\mu,{\mathbf{\eta}}}_T ),
\end{align*}
(so that $v(0,x_0) = \inf_{\mu\in\mathcal{M}} J( \mu )$).

Since ${\mathbf{\eta}}$ is geometric, we have 
\begin{align*}
\langle Dg( \bar X_T ), Y^\varepsilon_T \rangle &= \langle p_T,
Y^\varepsilon_T \rangle - \langle p_0, Y^\varepsilon_0 \rangle \\
&= - \int_0^T \langle Df(\bar X_r, \bar \mu_r), Y^\varepsilon_r \rangle dr +
\int_0^T \langle p(r), \left[ b(\bar X_r, \mu_r) - b(\bar X_r, \bar \mu_r) %
\right] 1_I(r) \rangle dr.
\end{align*}
Here, $Y^\varepsilon$ and $I$ are given as in Lemma \ref{lem:spikeVariation}.

Let any $u \in U$ be given. Let $\mu(t) \equiv u$. Let $t \in [0,T)$ and let 
$\varepsilon >0$ small enough such that $I_\varepsilon := [t,t+\varepsilon]
\subset [0,T]$. Then, combined with Lemma \ref{lem:spikeVariation} we get 
\begin{align*}
0 &\ge J(\mu^\varepsilon) - J(\bar \mu) \\
&= \langle Dg( \bar X_T ), Y^\varepsilon_T \rangle + \int_0^T \left[ \langle
Df( \bar x_r, \bar \mu_r ), Y^\varepsilon_r \rangle + \left\{ f( \bar X_r,
\mu_r ) - f( \bar X_r, \bar \mu_r ) \right\} 1_I(r) \right] dr +
o(\varepsilon) \\
&= - \int_0^T \langle Df(\bar X_r, \bar \mu_r), Y^\varepsilon_r \rangle dr +
\int_0^T \langle p(r), \left[ b(\bar X_r, \mu_r) - b(\bar X_r, \bar \mu_r) %
\right] 1_I(r) \rangle dr \\
&\qquad + \int_0^T \left[ \langle Df( \bar x_r, \bar \mu_r ),
Y^\varepsilon_r \rangle + \left\{ f( \bar X_r, \mu_r ) - f( \bar X_r, \bar
\mu_r ) \right\} 1_I(r) \right] dr + o(\varepsilon) \\
&= \int_t^{t+\varepsilon} \langle p(r), \left[ b(\bar X_r, u) - b(\bar X_r,
\bar \mu_r) \right] \rangle + f( \bar X_r, u ) - f( \bar X_r, \bar \mu_r )
dr + o(\varepsilon).
\end{align*}
Dividing by $\varepsilon$ and sending $\varepsilon \to 0$ yields, together
with the separability of the metric space, the desired result.
\end{proof}

\subsection{Pathwise stochastic control}

We can apply Theorem \ref{thm:roughHJ} to enhanced Brownian motion, i.e.
take $\mathbf{\eta }=\mathbf{B}\left( \omega \right) $, Brownian motion
enhanced with L\'{e}vy's stochastic area which constitutes for a.e. $\omega $
a geometric rough path. The (rough)pathwise unique solution to the RDE with
controlled drift, $X^{\mu,\mathbf{\eta }}|_{\mathbf{\eta =B}\left( \omega
\right) }$ then becomes a solution to the classical stochastic differential
equation (in Stratonovich sense) (Theorem \ref{Thm:RDEWithControlledDrift}). 

\begin{proposition}
\label{prop:measurable} Under the assumptions of Theorem \ref{thm:roughHJ},
the map 
\begin{equation*}
\omega \mapsto \left. \sup_{\mu \in\mathcal{M}} \left\{ \int_{t}^{T}f\left(
s,X^{\mu, \mathbf{\eta }},\mu_{s}\right) ds+g\left( X_{T}^{\mu,\mathbf{\eta }%
}\right) \right\} \right\vert _{\mathbf{\eta =B}\left( \omega \right) }
\end{equation*}
is measurable. In particular, the expected value of the pathwise
optimization problem,%
\begin{equation}
\bar{v}\left( t,x\right) =\mathbb{E}\left[ \left. \sup_{\mu\in\mathcal{M}%
}\left\{ \int_{t}^{T}f\left( s,X^{\mu,\mathbf{\eta }},\mu_{s}\right)
ds+g\left( X_{T}^{\mu,\mathbf{\eta }}\right) \right\} \right\vert _{\mathbf{%
\eta =B}\left( \omega \right) }\right]  \label{equ:Vbar}
\end{equation}%
is well-defined.
\end{proposition}

\begin{proof}
The lift into rough pathspace, $\omega \mapsto \mathbf{B}\left( \omega
\right) $, is measurable. $v^{\mathbf{\eta }}$ as element in BUC space
depends continuously (and hence: measurably) on the rough path $\mathbf{\eta 
}$. Conclude by composition.
\end{proof}

\begin{remark}
Well-definedness of such expressions was a non-trivial technical obstacle in
previous works on pathwise stochastic control; e.g. \cite{DB, Bu}. The use
of rough path theory allows to bypass this difficulty entirely.
\end{remark}

\begin{remark}
\label{rem:uinsig} Let us explain why we only consider the case where the
coefficient $\sigma(x)$ in front of the rough path is not controlled. It
would not be too difficult to make sense of RDEs 
\begin{align*}
dX &= b(t,X,u) dt + \sigma(X,u) d{\mathbf{\eta}}_t,
\end{align*}
assuming good regularity for $\sigma$ and $(u_s)_{s \geq 0}$ chosen in a
suitable class (for instance : $u$ piecewise constant, $u$ controlled by ${%
\mathbf{\eta}}$ in the Gubinelli sense,...). However, in most cases of
interest the control problem would degenerate, in the sense that we would
have 
\begin{align*}
v^{\mathbf{\eta }}\left( t,x\right) &= \sup_{\mu \in \mathcal{M}} \left\{
\int_{t}^{T}f\left( s,X_{s}^{t,x,\mu,{\mathbf{\eta}}},\mu_{s}\right)
ds+g\left( X_{T}^{t,x,\mu,{\mathbf{\eta}}}\right) \right\} \\
&= \int_t^T (\sup_{\mu,x} f(s,\mu,x)) ds + \sup_x g(x).
\end{align*}
The reason is that if $\sigma$ has enough $u$-dependence (for instance if $%
d=1$, $U$ is the unit ball in ${\mathbb{R}} ^e$ and $\sigma(x,u)=u$) and $%
\eta$ has unbounded variation on any interval (as is the case for typical
Brownian paths), the system can essentially be driven to reach any point
instantly.

In order to obtain nontrivial values for the problem, one would need the
admissible control processes to be uniformly bounded in some particular
sense (see e.g. \cite{MN08} in the Young case, where the $(\mu_s)$ need to
be bounded in some H\"older space), which is not very natural (for instance,
Dynamic Programming and HJB-type pointwise optimizations are then no longer
valid).
\end{remark}

\section{Duality results for classical stochastic control}

\label{sec:dualityResults}

We now link the expected value of the pathwise optimization problem, as
given in (\ref{equ:Vbar}), to the value function of the (classical)
stochastic control problem as exposed in \cite{Kr, FS}, 
\begin{equation}
V\left( t,x\right) :=\sup_{\nu \in \mathcal{A}}\mathbb{E}\left[
\int_{t}^{T}f\left( s,X^{t,x,\nu}_s,\nu_{s}\right) ds+g\left(
X_{T}^{t,x,\nu}\right) \right] .  \label{defV}
\end{equation}%
Here, the $\sup$ is taken over all admissible (in particular: adapted)
controls $\nu$ on $\left[ t,T\right]$. 
There are well-known assumptions under which $V$ is a classical (see \cite%
{Kr}) resp. viscosity (see \cite{FS}) solution to the HJB\ equation, i.e.
the non-linear terminal value problem 
\begin{eqnarray*}
-\partial_{t}V-F\left( t,x,DV,D^{2}V\right) &=&0 \\
V\left( T,\cdot \right) &=&g;
\end{eqnarray*}%
uniqueness holds in suitable classes. In fact, assume the dynamics\footnote{%
Again, it would be straightforward to include explicit time dependence in
the vector fields $b$, $\sigma_i$.} 
\begin{align}  \label{eq:controlledSDE}
\begin{split}
dX^{s,x,\nu}_t &= b\left( X^{s,x,\nu}_t,\nu_t\right) dt+ \sum_{i=1}^d
\sigma_i \left( X^{s,x,\nu}_t\right) \circ dB^i_t, \qquad X^{s,x,\nu}_s = x,
\\
&= \tilde b\left( X^{s,x,\nu}_t,\nu_t\right) dt+ \sum_{i=1}^d \sigma_i
\left( X^{s,x,\nu}_t\right) dB^i_t, \qquad X^{s,x,\nu}_s = x,
\end{split}%
\end{align}
where $\tilde b(x,u) = b(x,u) + \frac{1}{2} \sum_{i=1}^d (\sigma_i \cdot
D\sigma_i)(x)$ is the corrected drift. Then the equation is semilinear of
the form 
\begin{equation}  \label{equ:HJB_semilinear}
\begin{split}
-\partial_{t}V - \tilde H\left(t,x,DV\right) - LV &= 0, \\
V\left( T, \cdot \right) &= g.
\end{split}%
\end{equation}
where 
\begin{equation*}
L V = \frac{1}{2} \func{Trace}[ (\sigma \sigma^T) D^2 V ]
\end{equation*}
and $\tilde H$ is given by $\tilde H\left( t, x, p\right) = \sup_{u}\left\{
\left\langle \tilde b\left( x, u\right) , p\right\rangle +f\left(
t,x,u\right) \right\}$.

Let us also write 
\begin{align*}
L^{u} V=\tilde b\left( \cdot ,u\right) DV + L V, \qquad u \in U.
\end{align*}

Denote by $\mathcal{A}$ the class of progressively measurable controls $\nu:
\Omega \times [0,T] \to U$. As before $\mathcal{M}$ is the class of
measurable function $\mu: [0,T] \to U$, with the topology of convergence in
measure (with respect to $dt$).


\begin{theorem}
\label{thm:generalDuality} Let $\mathcal{Z}_ {\mathcal{F}} $ be the class of
all mappings $z: \mathcal{C}^{0,\alpha} \times \mathcal{M} \to {\mathbb{R}} $
such that

\begin{itemize}
\item $z$ is bounded, measurable and continuous in ${\mathbf{\eta}} \in 
\mathcal{C}^{0,\alpha}$ uniformly over $\mu \in \mathcal{M}$ 

\item ${\mathbb{E}} [ z(\mathbf{B}, \nu) ] \ge 0$, if $\nu$ is adapted
\end{itemize}

Let $b: {\mathbb{R}} ^e \times U \to {\mathbb{R}} ^e$ be continuous and let $%
b(\cdot,u) \in \func{Lip}^1( {\mathbb{R}} ^e)$ uniformly in $u \in U$. Let $%
\sigma_1,\dots,\sigma_d \in \func{Lip}^\gamma( {\mathbb{R}} ^e)$, for some $%
\gamma > 2$, $g \in BUC( {\mathbb{R}} ^e)$ and $f: [0,T] \times {\mathbb{R}}
^e \times U \to {\mathbb{R}} $ bounded, continuous and locally uniformly
continuous in $t,x$, uniformly in $u$. Then 
we have 
\begin{align*}
V(t,x) = \inf_{z \in \mathcal{Z}_ {\mathcal{F}} } {\mathbb{E}} \left[ \left.
\sup_{\mu \in \mathcal{M}} \left\{ \int_t^T f(r,X^{t,x,\mu,{\mathbf{\eta}}%
}_r,\mu_r) dr + g(X^{t,x,\mu,{\mathbf{\eta}}}_T) + z( {\mathbf{\eta}}, \mu )
\right\} \right\vert_{{\mathbf{\eta}} = \mathbf{B}(\omega)} \right].
\end{align*}
Where $\mathbf{B}$ denotes the Stratonovich lift of Brownian motion to a
geometric rough path and $X^{t,x,\mu,{\mathbf{\eta}}}$ is the solution to
the RDE with controlled drift (Theorem \ref{Thm:RDEWithControlledDrift}) 
\begin{equation*}
dX^{t,x,\mu,\mathbf{\eta}}=b\left( X^{t,x,\mu,\mathbf{\eta}},\mu\right)
dt+\sigma \left( X^{t,x,\mu,\mathbf{\eta}}\right) d\mathbf{\eta}, \quad
X^{t,x,\mu,\mathbf{\eta}}_t = x.
\end{equation*}
\end{theorem}

\begin{remark}
Every choice of admissible control $\nu\in\mathcal{A}$ in (\ref{defV}) leads
to a lower bound on the value function (with equality for~$\nu=\nu^{\ast},$
the optimal control). In the same spirit, every choice $z$ leads to an upper
bound. 
There is great interest in such duality results, as they help to judge how
much room is left for policy improvement. The result is still too general
for this purpose and therefore it is an important question, discussed below,
to understand whether duality still holds when restricting to some concrete
(parametrized) subsets of $\mathcal{Z}_ {\mathcal{F}} $.
\end{remark}

\begin{proof}
We first note, that the supremum inside the expectation is continuous (and
hence measurable) in ${\mathbf{\eta}}$, which follows by the same argument
as in the proof of Theorem \ref{thm:roughHJ}. Since it is also bounded, the
expectation is well-defined.

Recall that $X^{t,x,\nu}$ is the solution to the (classical) controlled SDE
and that $X^{t,x,\mu,{\mathbf{\eta}}}$ is the solution to the controlled
RDE. Let $z \in \mathcal{Z}_ {\mathcal{F}} $. Then, using Theorem \ref%
{Thm:RDEWithControlledDrift} to justify the step from second to third line, 
\begin{align*}
V(t,x) &= \sup_{\nu \in \mathcal{A}} {\mathbb{E}} \left[ \int_t^T f(s,
X^{t,x,\nu}_s, \nu_s) ds + g(X^{t,x,u}_T) \right] \\
&\le \sup_{\nu \in \mathcal{A}} {\mathbb{E}} \left[ \int_t^T f(s,
X^{t,x,\nu}_s, \nu_s) ds + g(X^{t,x,\nu}_T) + z(\mathbf{B}, \nu) \right] \\
&= \sup_{\nu \in \mathcal{A}} {\mathbb{E}} \left[ \left\{ \int_t^T f(s,
X^{t,x,\mu,{\mathbf{\eta}}}_s, \mu_s) ds + g(X^{t,x,\mu,{\mathbf{\eta}}}_T)
+ z({\mathbf{\eta}}, \mu) \right\}_{\mu=\nu, {\mathbf{\eta}}=\mathbf{B}} %
\right] \\
&\le {\mathbb{E}} \left[ \sup_{\mu\in\mathcal{M}} \left\{ \int_t^T f(s,
X^{t,x,\mu,{\mathbf{\eta}}}_s, \mu_s) ds + g(X^{t,x,\mu,{\mathbf{\eta}}}_T)
+ z({\mathbf{\eta}}, \mu) \right\}_{{\mathbf{\eta}}=\mathbf{B}} \right],
\end{align*}

And to show equality, let 
\begin{align*}
z^*({\mathbf{\eta}},\mu) := V(t,x) - \int_t^T f(s, X^{t,x,\mu,{\mathbf{\eta}}%
}_s, \mu_s) ds + g(X^{t,x,\mu,{\mathbf{\eta}}}_T).
\end{align*}
Then $z^* \in \mathcal{Z}_ {\mathcal{F}} $ and equality is attained.
\end{proof}

%

\subsection{Example I, inspired by the discrete-time results of Rogers 
\protect\cite{R07}}

We now show that Theorem \ref{thm:generalDuality} still holds with penalty
terms based on martingale increments.

\begin{theorem}
\label{thm:dualityUsingVF} Under the (regularity) assumptions of Theorem \ref%
{thm:generalDuality} we have 
\begin{align*}
V\left( t,x\right) &=\inf_{h\in C^{1,2}_b}\mathbb{E}\left[ \left. \sup_{\mu
\in \mathcal{M} }\left\{ \int_{t}^{T}f\left( s,X^{t,x,\mu,\mathbf{\eta }%
},\mu_{s}\right) ds+g\left( X_{T}^{t,x,\mu,\mathbf{\eta }}\right)
-M_{t,T}^{t,x,\mu,\mathbf{\eta},h}\right\} \right\vert _{\mathbf{\eta =B}%
\left( \omega \right) }\right],
\end{align*}
where 
\begin{equation*}
M_{t,T}^{t,x,\mu,{\mathbf{\eta}},h} := h\left( T,X_{T}^{t,x,\mu,{\mathbf{\eta%
}} }\right) -h\left( t, X_{t}^{t,x,\mu,{\mathbf{\eta}}} \right)
-\int_{t}^{T}\left( \partial _{s}+L^{\mu_s}\right) h\left(s,X_{s}^{t,x,\mu,{%
\mathbf{\eta}} }\right) ds.
\end{equation*}
That is, Theorem \ref{thm:generalDuality} still holds with $\mathcal{Z}_ {%
\mathcal{F}} $ replaced by the set $\{ z: z({\mathbf{\eta}}, \mu) =
M_{t,T}^{t,x,\mu,{\mathbf{\eta}},h}, h \in C^{1,2}_b \}$. Moreover, if $V$ $%
\in$ $C^{1,2}_b$ the infimum is achieved at $h^{\ast}=V$.
\end{theorem}


\begin{proof}
We have 
\begin{align*}
&V(t,x) \\
&\le \inf_{h \in C^{1,2}_b} {\mathbb{E}} \left[ \left. \sup_{\mu \in 
\mathcal{M}} \left\{ \int_t^T f(s, X^{t,x,\mu,{\mathbf{\eta}}}_s, \mu_s) ds
+ g( X^{t,x,\mu,{\mathbf{\eta}}}_T ) - M_{t,T}^{t,x,\mu,{\mathbf{\eta}},h}
\right\} \right\vert_{{\mathbf{\eta}}=\mathbf{B}} \right] \\
&= \inf_{h \in C^{1,2}_b} \Bigl( h(t,x) \\
&\qquad + {\mathbb{E}} \left[ \left. \sup_{\mu \in \mathcal{M}} \left\{
\int_t^T f(s, X^{t,x,\mu,{\mathbf{\eta}}}_s, \mu_s) + (\partial_s +
L^{\mu_s}) h(s, X^{t,x,\mu,{\mathbf{\eta}}}_s) ds + g( X^{t,x,\mu,{\mathbf{%
\eta}}}_T ) - h(T, X^{t,x,\mu,{\mathbf{\eta}}}_T ) \right\} \right\vert_{{%
\mathbf{\eta}}=\mathbf{B}} \right] \Bigr) \\
&\le \inf_{h \in C^{1,2}_b} \left( h(t,x) + \int_t^T \sup_{x \in {\mathbb{R}}
^e, u \in U} \left[ f(s, x, u) + (\partial_s + L^{u}) h(s, x) \right] ds +
\sup_{x \in {\mathbb{R}} ^e} \left[ g( x ) - h(T, x ) \right] \right) \\
&\le \inf_{h \in S^{+}_{s}} \left( h(t,x) + \int_t^T \sup_{x \in {\mathbb{R}}
^e, u \in U} \left[ f(s, x, u) + (\partial_s + L^{u}) h(s, x) \right] ds +
\sup_{x \in {\mathbb{R}} ^e} \left[ g( x ) - h(T, x ) \right] \right) \\
&\le \inf_{h \in S^{+}_s} h(t,x).
\end{align*}
where the first inequality follows from (the proof of) Theorem \ref%
{thm:generalDuality} and $S^{+}_s$ denotes the class of smooth super
solutions of the HJB equation.

But in fact the infimum of smooth supersolutions is equal to the viscosity
solution $V$, all inequalities are actually equalities and the result
follows. This can be proved via a technique due to Krylov \cite{Kr00} which
he called "shaking the coefficients". For the reader's convenience let us
recall the argument.

Extending by continuity $\tilde b, \sigma$ and $f$ to $t$ $\in$ $(-\infty,
\infty)$, define for $\varepsilon$ $>$ $0$, 
\begin{eqnarray*}
F^\varepsilon(t,x,p,X) &:=& \sup_{u \in U, |s|, |e| \leq \varepsilon } \left[
\tilde \langle b(x+e,u), p \rangle + \frac{1}{2} Tr( \sigma \sigma^T(x+e) X)
+ f(t+s,x+e,u) \right],
\end{eqnarray*}
and consider $V^\varepsilon$ the unique viscosity solution to 
\begin{eqnarray*}
\left\{ 
\begin{array}{ccc}
- \frac{\partial V^\varepsilon}{\partial t} -
F^\varepsilon(t,x,DV^\varepsilon, D^2V^\varepsilon) & = & 0, \\ 
V^\varepsilon(T,\cdot) & = & g.%
\end{array}
\right.
\end{eqnarray*}

By (local) uniform continuity of $b$, $\sigma$, $f$ one can actually show
that $V \to V^\varepsilon$ as $\varepsilon$ $\to$ $0$, locally uniformly.
This can be done for instance by interpreting $V^\varepsilon$ as the value
function of a stochastic control problem.

Now take some smoothing kernel $\rho_\varepsilon$ with $\int_{ {\mathbb{R}}
^{e+1}} \rho_\varepsilon = 1$ and $\func{supp}(\rho_\varepsilon) \subset
[-\varepsilon, \varepsilon]^{e+1}$, and define $V_\varepsilon :=
V^\varepsilon \ast \rho_\varepsilon$. Clearly by definition of $%
F^\varepsilon $, for each $|s|, |e| \leq \varepsilon$, $V^\varepsilon( \cdot
-s, \cdot - e) $ is a supersolution to the HJB equation $-\partial_t V -
F(t,x,DV,D^2V) = 0$. Since $F$ is convex in $(DV,D^2V)$ it follows that 
\begin{eqnarray*}
V_\varepsilon &=& \int_{[-\varepsilon, \varepsilon]^{e+1}}
V^\varepsilon(\cdot-s, \cdot - e) \rho_\varepsilon(s,e) ds de
\end{eqnarray*}
is again a (smooth) supersolution (for the details see the appendix in \cite%
{BJ}). Finally it only remains to notice that $| V - V_\varepsilon| $ $\leq$ 
$| V - V \ast \rho_\varepsilon | + |(V - V^\varepsilon) \ast
\rho_\varepsilon|$ $\to$ $0$ (locally uniformly).
\end{proof}

\begin{remark}
Note that 
\begin{equation*}
V^h(t,x) := \mathbb{E}\left[ \left. \sup_{\mu \in \mathcal{M} }\left\{
\int_{t}^{T}f\left( s,X^{t,x,\mu,\mathbf{\eta }},\mu_{s}\right) ds+g\left(
X_{T}^{t,x,\mu,\mathbf{\eta }}\right) -M_{t,T}^{t,x,\mu,\mathbf{\eta}%
,h}\right\} \right\vert _{\mathbf{\eta =B}\left( \omega \right) }\right],
\end{equation*}
\underline{for fixed $x,t$}, is precisely of the form \ref{equ:Vbar} with $f$
resp. $g$ replaced by $\tilde{f}$ resp. $\tilde{g}$, given by%
\begin{eqnarray*}
\tilde{f}\left( s,\cdot ,\mu\right) &=&f\left( s,\cdot ,\mu\right) +\left(
\partial _{s}+L^{\mu}\right) h\left( s,\cdot \right) , \\
\tilde{g}\left( \cdot \right) &=&g\left( \cdot \right) +h\left( T,\cdot
\right) -h\left( t,x\right) .
\end{eqnarray*}%
%
%
The point is that the inner pathwise optimization falls directly into the
framework of Section \ref{sec:deterministicControl}.
\end{remark}

\begin{remark}
\label{rem:RogersInt} For ${\mathbf{\eta }}$ a geometric rough path, we may
apply the chain rule to $h(s,X_{s})$ and obtain 
\begin{equation*}
h(T,X_{T})-h(t,x)=\int_{t}^{T}\langle Dh(s,X_{s}),b(s,X_{s},\mu
_{s})ds+\sigma (X_{s})d{\mathbf{\eta }}_{s}\rangle .
\end{equation*}%
It follows that the penalization may also be rewritten in a (rough) integral
form 
\begin{equation*}
M_{t,T}^{t,x,\mu ,\mathbf{\eta },h}=\;\int_{t}^{T}\langle Dh(s,X_{s}),\sigma
(X_{s})d{\mathbf{\eta }}_{s}\rangle +\int_{t}^{T}\left\{ \langle (b-\tilde{b}%
)(s,X_{s},\mu _{s}),Dh(s,X_{s})\rangle +\func{Tr}[(\sigma \sigma
^{T})D^{2}h](s,X_{s})\right\} ds.
\end{equation*}%
Note that for ${\mathbf{\eta }}=\mathbf{B}$ and adapted $\nu $, this is just
the It\^{o} integral $\int_{t}^{T}\langle Dh(s,X_{s}),\sigma
(X_{s})dB_{s}\rangle $.
\end{remark}

\begin{remark}
If one were to try anticipating stochastic calculus, in the spirit \cite{DB}%
, to implement Roger's duality in continuous time, then - leaving aside all
other technical (measurability) issues that have to be dealt with - more
regularity on the coefficient will be required. This is in stark contrast to
the usual understanding in SDE theory that rough paths require more
regularity than It\^{o} theory.
\end{remark}

\begin{example}
From example \ref{ex:additive} we can see that in some special cases this
method gives explicit upper bounds. Assume :

\begin{itemize}
\item additive noise ($\sigma \equiv Id$),

\item state-independent drift $b \equiv b(u)$,

\item running gain $f(s,x,u) = f^0(u) + \nabla h(x) \cdot b(u)$, with $h$
subharmonic ($\Delta h \geq 0$).
\end{itemize}

Then for the penalty corresponding to $h(t,x)=h(x)$, the inner optimization
problem is given by 
\begin{align*}
& \sup_{\mu \in \mathcal{M} }\left\{ \int_{t}^{T}(f^0\left(\mu_{s}\right) +
\langle \nabla h(X_{s}^{t,x,\mu,\mathbf{\eta } }), b(\mu_s) \rangle
)ds+g\left( X_{T}^{t,x,\mu,\mathbf{\eta }}\right) -\int_t^T(\langle \nabla
h(X_{s}^{t,x,\mu,\mathbf{\eta }}), b(\mu_s) \rangle + \frac{1}{2} \Delta
h(X_{s}^{t,x,\mu,\mathbf{\eta }}))ds \right\} \\
&\leq \sup_{\mu \in \mathcal{M} }\left\{
\int_{t}^{T}f^0\left(\mu_{s}\right)ds+(g-h)\left( X_{T}^{t,x,\mu,\mathbf{%
\eta }}\right) \right\} \;\;\;= \;\;\; V^{0,h}(t, x + \eta_T - \eta_t),
\end{align*}
where $V^{0,h}$ is the value function to the standard control problem 
\begin{align*}
V^0(t,x) &= \sup_{\mu \in \mathcal{M}} \left\{ \int_t^T f^0(\mu_s) ds +
(g-h)\left( x + \int_t^T \mu_s ds \right) \right\}.
\end{align*}
From Theorem \ref{thm:dualityUsingVF}, we then have the upper bound 
\begin{align*}
V(t,x) &\leq h(x) + {\mathbb{E}} \left[V^{0,h} \left(t, x + B_T -B_t \right) %
\right].
\end{align*}
\end{example}

\begin{remark}
As in Remark \ref{rem:uinsig}, one can wonder how Theorem \ref%
{thm:dualityUsingVF} could translate in the case where $\sigma$ depends on $%
u $. As mentioned in that remark, under reasonable conditions on $\sigma$
the control problem degenerates so that \emph{for any choice of $h$}, say
for piecewise-constant controls $\mu$, we can expect that 
\begin{align*}
& {\mathbb{E}} \left[ \left. \sup_{\mu} \left\{ \int_t^T f(s, X^{t,x,\mu,{%
\mathbf{\eta}}}_s, \mu_s) + (\partial_t + L^{\mu_s}) h(s, X^{t,x,\mu,{%
\mathbf{\eta}}}_s) ds + g( X^{t,x,\mu,{\mathbf{\eta}}}_T ) - h(T, X^{t,x,\mu,%
{\mathbf{\eta}}}_T ) \right\} \right\vert_{{\mathbf{\eta}}=\mathbf{B}} %
\right] \Bigr) \\
&= \int_t^T \sup_{x \in {\mathbb{R}} ^e, u \in U} \left[ f(s, x, u) +
(\partial_t + L^{u}) h(s, x) \right] ds + \sup_{x \in {\mathbb{R}} ^e} \left[
g( x ) - h(T, x ) \right].
\end{align*}

In other words there is nothing to be gained from considering the
(penalized) pathwise optimization problem, as we always get 
\begin{align*}
V(t,x) &\leq h(t,x) + \int_t^T \sup_{x \in {\mathbb{R}} ^e, u \in U} \left[
f(s, x, u) + (\partial_t + L^{u}) h(s, x) \right] ds + \sup_{x \in {\mathbb{R%
}} ^e} \left[ g( x ) - h(T, x ) \right]
\end{align*}
which is in fact clear from a direct application of It\^o's formula (or
viscosity comparison). 
\end{remark}

\subsection{Example II, inspired by Davis--Burstein \protect\cite{DB}}

Under certain concavity assumptions, it turns out that linear penalization
is enough.

\begin{theorem}
\label{thm:dbDuality} 
Let $g$ be as in Theorem \ref{thm:generalDuality} and assume $f=0$;
furthermore make the (stronger) assumption that $b \in C^5_b$, $\sigma \in
C^5_b$, 
$\sigma \sigma^T > 0$, and that \eqref{defV} has a feedback solution $u^*$
which is continuous, $C^1$ in $t$ and $C^4_b$ in $x$, taking values in the
interior of $U$. Assume that $U$ is a compact convex subset of ${\mathbb{R}}
^n$,

Let $Z^{t,x,{\mathbf{\eta}}}$ be the solution starting from $x$ at time $t$
to : 
\begin{eqnarray}  \label{eq:dynZ}
dZ &=& b\left( Z,u^*(t,Z)\right) dt+\sigma \left( Z\right) d{\mathbf{\eta}}
- b_u(Z,u^*(t,Z)) u^*(t,Z) dt,
\end{eqnarray}
let $W(t,x):=W(t,x;{\mathbf{\eta}}):=g(Z_T^{t,x,{\mathbf{\eta}}})$ and
assume that 
\begin{eqnarray}
\forall (t,x), \;\;\; u \mapsto \left\langle b(x,u), DW(t,x;\mathbf{B})
\right\rangle \;\;\;\; \mbox{ is strictly concave, a.s.}  \label{eq:convDB}
\end{eqnarray}
Then 
\begin{align*}
V(t,x) &= \inf_{\lambda \in A} {\mathbb{E}} [ \left. \sup_{\mu \in \mathcal{M%
}} \left\{ \int_t^T f(r,X^{t,x,\mu,{\mathbf{\eta}}}_r,\mu_r) dr +
g(X^{t,x,\mu,{\mathbf{\eta}}}_T) + \int_t^T \langle \lambda(r,X^{t,x,\mu,{%
\mathbf{\eta}}}_r, {\mathbf{\eta}}), \mu_r \rangle dr \right\} \right\vert_{{%
\mathbf{\eta}} = \mathbf{B}(\omega)} ].
\end{align*}

Where 
$A$ is the class of all $\lambda: [0,T] \times {\mathbb{R}} ^e \times 
\mathcal{C}^{0,\alpha} \to {\mathbb{R}} ^d$ such that

\begin{itemize}
\item $\lambda$ is bounded and uniformly continuous on bounded sets 

\item $\lambda$ is future adapted, i.e. for any fixed $t,x$, $\lambda(t,x, 
\mathbf{B} ) \in {\mathcal{F}} _{t,T}$

\item ${\mathbb{E}} [ \lambda(t,x,\mathbf{B}) ] = 0$ for all $t,x$.
\end{itemize}

That is, Theorem \ref{thm:generalDuality} still holds with $\mathcal{Z}_ {%
\mathcal{F}} $ replaced by the set 
\begin{align*}
\{ z: z({\mathbf{\eta}}, \mu) = \int_t^T \langle \lambda(s,X^{t,x,\mu,{%
\mathbf{\eta}}}_r,{\mathbf{\eta}}), \mu_s \rangle ds, \lambda \in A \}.
\end{align*}
Moreover the infimum is achieved with $\lambda^*(t,x,{\mathbf{\eta}}) :=
b_u^T(t, u^*(t,x)) DW(t,x;{\mathbf{\eta}})$.
\end{theorem}

\begin{remark}
The concavity assumption is difficult to verify for concrete examples. It
holds for the linear quadratic case, which we treat in Section \ref%
{ss:explicitComputations}.
\end{remark}

\begin{remark}
\label{rem:DBrunningcost} The case of running cost $f$ is, as usual, easily
covered with this formulation. Indeed, let the optimal control problem be
given as 
\begin{align*}
dX &= b(X,\nu) dt + \sigma(X) \circ dW, \\
V(t,x) &= \sup_\nu {\mathbb{E}} [ \int_t^T f(X,\nu) dr + g(X_T) ].
\end{align*}
Define the new component 
\begin{align*}
dX^{d+1}_t = f(X,u) dt, \qquad X^{d+1}_t = x.
\end{align*}
Then the theorem yields that the penalty 
\begin{align*}
\lambda^*(t,x) &:= (b_u, f_u) \cdot (D_{x_{1\dots d}} g(Z_T) + D_{x_{1\dots
d}} Z^{e+1}_T, D_{x_{e+1}} Z^{e+1}_T) \\
&= (b_u, f_u) \cdot (D_{x_{1\dots d}} g(Z_T) + D_{x_{1\dots d}} Z^{e+1}_T,
1).
\end{align*}
is optimal, where 
\begin{align*}
dZ &= b(Z,u^*) dt + \sigma(Z) d{\mathbf{\eta}} - b_u(Z,u^*) u^* dt, \\
dZ^{e+1} &= f(Z,u^*) dt - f_u(Z,u^*) u^* dt.
\end{align*}
\end{remark}


\begin{proof}
From (the proof of) Theorem \ref{thm:generalDuality} we know $V(t,x) \le
\inf_{\lambda \in A} {\mathbb{E}} [ \dots ]$. The converse direction is
proven in \cite{DB} by using \footnote{%
The paper of Davis--Burstein predates rough path theory and their proof
relies on anticipating stochastic calculus.} 
\begin{align*}
\lambda^*(t,x,{\mathbf{\eta}}) = b_u^T(t, u^*(t,x)) DW(t,x).
\end{align*}
For the reader's convenience we provide a sketch of the argument below. 

\end{proof}


\begin{proof}[Sketch of the Davis-Burstein argument]
We have assumed that the optimal control for the stochastic problem %
\eqref{defV} is given in feedback form by $u^*(t,x)$. Write $X^{t,x,*} :=
X^{t,x,u^*}$.

Recall that $Z^{t,x,{\mathbf{\eta}}}$ is the solution starting from $x$ at
time $t$ to : 
\begin{eqnarray*}
dZ &=& b\left( Z,u^*(t,Z)\right) dt+\sigma \left( Z\right) d{\mathbf{\eta}}
- b_u(Z,u^*(t,Z)) u^*(t,Z) dt.
\end{eqnarray*}
Assume that $W(t,x):=W(t,x;{\mathbf{\eta}})=g(Z_T^{t,x,{\mathbf{\eta}}})$ is
a (viscosity) solution to the rough PDE 
\begin{eqnarray}
- \partial_t W - \left\langle b(x,u^*(t,x)) - b_u(x,u^*(t,x))u^*(t,x), DW
\right\rangle - \left\langle \sigma \left( x\right), DW\right\rangle \dot{{%
\mathbf{\eta}}} &=& 0,  \label{eq:pdeWZ}
\end{eqnarray}
and assume that $W$ is differentiable in $x$.

We assumed that 
\begin{eqnarray*}
\forall (t,x), \;\;\; u \mapsto \left\langle b(x,u), DW(t,x) \right\rangle
\;\;\;\; \mbox{ is strictly concave.}
\end{eqnarray*}
It then follows that 
\begin{eqnarray}
\left\langle b(x,u^*(t,x)) - b_u(x,u^*(t,x))u^*(t,x), DW \right\rangle &=&
\sup_{u \in U} \left\langle b(x,u) - b_u(x,u^*(t,x))u, DW \right\rangle.
\label{eq:supu*}
\end{eqnarray}
Because of \eqref{eq:supu*} the PDE \eqref{eq:pdeWZ} may be rewritten as 
\begin{eqnarray*}
&& - \partial_t W - \left\langle b(x,u^*(t,x)), DW \right\rangle -
\left\langle u^*(t,x), \lambda^*(t,x;{\mathbf{\eta}})\right\rangle -
\left\langle \sigma \left( x\right), DW\right\rangle \dot{{\mathbf{\eta}}} \\
&=& - \partial_t W - \sup_{u \in U} \left\{ \left\langle b(x,u), DW
\right\rangle - \left\langle u, \lambda^*(t,x;{\mathbf{\eta}})\right\rangle
\right\} - \left\langle \sigma \left( x\right), DW\right\rangle \dot{{%
\mathbf{\eta}}} \\
&=& 0.
\end{eqnarray*}
By verification it follows that actually $W$ is also the value function of
the problem with penalty $\lambda^*$, and the optimal control is given by $%
u^*$, i.e. 
\begin{eqnarray*}
W(t,x) &=& W(t,x;{\mathbf{\eta}}) = \sup_{\mu \in \mathcal{M}} \left[
g(X_T^{t,x,\mu,{\mathbf{\eta}}}) - \int_t^T \left\langle
\lambda^*(s,X_s^{t,x,\mu,{\mathbf{\eta}}};{\mathbf{\eta}}) , \mu_s
\right\rangle ds \right] \\
&=& g(X_T^{t,x,u^*,{\mathbf{\eta}}}) - \int_t^T \left\langle
\lambda^*(s,X_s^{t,x,u^*,{\mathbf{\eta}}};{\mathbf{\eta}}) ,
u^*(s,X_s^{t,x,u^*,{\mathbf{\eta}}}) \right\rangle ds \\
\end{eqnarray*}
Then, by Theorem \ref{Thm:RDEWithControlledDrift} we have (if the convexity
assumption \eqref{eq:convDB} is satisfied a.s. by ${\mathbf{\eta}}=\mathbf{B}%
(\omega)$) 
\begin{align*}
W(t,x;\mathbf{B}) = g(X_T^{t,x,*}) - \int_t^T \left\langle
\lambda^*(s,X_s^{t,x,*};\mathbf{B}) , u^*(s,X_s^{t,x,*}) \right\rangle ds.
\end{align*}
It follows in particular that for the original stochastic control problem 
\begin{eqnarray*}
V(t,x)\;=\; \sup_{\nu \in \mathcal{A}} {\mathbb{E}} \left[ g(X_T^{t,x,\nu}) %
\right] &=& {\mathbb{E}} \left[g(X_T^{t,x,*})\right] \\
&=& {\mathbb{E}} \left[g(X_T^{t,x,*}) - \int_t^T \left\langle
\lambda^*(s,X_s^{t,x,*};\mathbf{B}) , u^*(s,X_s^{t,x,*}) \right\rangle ds%
\right] \\
&=& {\mathbb{E}} \left[ \sup_{\mu \in \mathcal{M}} \left\{ g(X_T^{t,x,\mu,{%
\mathbf{\eta}}}) - \int_t^T \left\langle \lambda^*(s,X_s^{t,x,\mu,{\mathbf{%
\eta}}};{\mathbf{\eta}}) , \mu_s \right\rangle ds \right\}|_{{\mathbf{\eta}}=%
\mathbf{B}}\right].
\end{eqnarray*}
Here we have used that $\lambda^*(t,x,\mathbf{B})$ is future adapted and ${%
\mathbb{E}} [ \lambda^*(t,x,\mathbf{B}) ] = 0\ \forall t,x.$ which is shown
on p. 227 in \cite{DB}. 

\end{proof}

\begin{remark}
The two different penalizations presented above are based on verification
arguments for respectively the stochastic HJB equation and the (rough)
deterministic HJB equation. One can then also try to devise an approach
based on Pontryagin's maximum principles (both stochastic and
deterministic). While this is technically possible, the need to use
sufficient conditions in the rough PMP means that it can only apply to the
very specific case where $\sigma$ is affine in $x$, and in consequence we
have chosen not to pursue this here. 
\end{remark}

\subsection{Explicit computations in LQC problems}

\label{ss:explicitComputations}

We will compare the two optimal penalizations in the case of a linear
quadratic control problem (both for additive and multiplicative noise).

\subsubsection{LQC with additive noise}

The dynamics are given by \footnote{%
This equation admits an obvious pathwise SDE solution (via the ODE satisfied
by $X-B$) so that, strictly speaking, there is no need for rough paths here.}
\begin{eqnarray}
dX &=& (MX + N\nu) dt + dB_t
\end{eqnarray}
and the optimization problem is given by 
\begin{eqnarray}
V(t,x) &= \sup_{\nu \in \mathcal{A}} {\mathbb{E}} \left[ \frac{1}{2}\int_t^T
(\langle Q X_s, X_s \rangle + \langle R \nu_s, \nu_s\rangle) ds + \frac{1}{2}%
\langle G X_T, X_T \rangle \right].
\end{eqnarray}

This problem admits the explicit solution (see e.g. Section 6.3 in \cite{YZ}%
) 
\begin{eqnarray}
V(t,x) &=& \frac{1}{2}\langle P(t) x, x \rangle + \frac{1}{2}\int_t^T
Tr(P(s)) ds,
\end{eqnarray}
where $P$ is the solution to the matrix Riccati equation 
\begin{eqnarray*}
P(T) &=& G, \\
P'(t) &=& - P(t)M - {}^t M P(t) + P N R^{-1} {}^tN P(t) - Q,
\end{eqnarray*}
and the optimal control is then given in feedback form by 
\begin{eqnarray*}
\nu^*(t,x) &=& - R^{-1} {}^t N P(t) x.
\end{eqnarray*}

%
%
%

\begin{proposition}
For this LQ control problem the optimal penalty corresponding to Theorem \ref%
{thm:dbDuality} is given by 
\begin{align*}
z^{1}({\mathbf{\eta}},\mu) & = - \int_t^T \langle \lambda^{2}(s;{\mathbf{\eta%
}}), \mu_s \rangle ds,
\end{align*}
where 
\begin{align*}
\lambda^{1}(t; {\mathbf{\eta}}) &= - {}^t N \int_t^T e^{{}^t M (s-t)} P(s) d{%
\mathbf{\eta}}_s.
\end{align*}

The optimal penalty corresponding to Theorem \ref{thm:dualityUsingVF} is
given by 
\begin{align*}
z^{2}({\mathbf{\eta}},\mu) & = z^{1}({\mathbf{\eta}},\mu) + \gamma^R({%
\mathbf{\eta}}),
\end{align*}
where 
\begin{align*}
\gamma^R({\mathbf{\eta}}) &=\int_t^T \langle P(s) X_s^0, d{\mathbf{\eta}}_s
\rangle - \frac{1}{2} \int_t^T Tr(P(s))ds,
\end{align*}
$X^0$ denoting the solution to the RDE $dX = MX + d {\mathbf{\eta}}$
starting at $(t,x)$. 5 In particular, these two penalizations are equal
modulo a random constant (not depending on the control) with zero
expectation.
\end{proposition}

\begin{proof}
The formula for $z^{1}$ is in fact already computed in \cite[sec. 2.4]{DB},
so that it only remains to do the computation for the Rogers penalization.

It follows from Remark \ref{rem:RogersInt} that 
\begin{eqnarray*}
M^{t,x,\mu,{\mathbf{\eta}},V}_{t,T} &=& \int_t^T \langle DV(s,X_s), d{%
\mathbf{\eta}}_s \rangle - \frac{1}{2} \int_t^T Tr(D^2V(s,X_s))ds \\
&=& \int_t^T \langle P(s)X_s^\mu, d{\mathbf{\eta}}_s \rangle - \frac{1}{2}
\int_t^T Tr(P(s))ds \\
&=& \int_t^T \langle P(s)(X_s^0 + \int_0^s e^{M(s-r)} N \mu_r dr), d{\mathbf{%
\eta}}_s \rangle - \frac{1}{2} \int_t^T Tr(P(s))ds \\
&=& \int_t^T \langle \mu_r, ( {}^t N \int_r^T e^{{}^tM(s-r)} P(s)d{\mathbf{%
\eta}}_s) \rangle dr + \int_t^T \langle P(s)X_s^0, d{\mathbf{\eta}}_s
\rangle - \frac{1}{2}\int_t^T Tr(P(s))ds.
\end{eqnarray*}

Hence we see that this penalization can be written as $z^2 = z^{1} +
\gamma^R({\mathbf{\eta}})$, where $\gamma^R({\mathbf{\eta}})$ does not
depend on the chosen control. One can check immediately that ${\mathbb{E}}
[\gamma^R({\mathbf{\eta}})\vert_{{\mathbf{\eta}} = \mathbf{B}(\omega)}]=0$.

%
%
%
\end{proof}

\subsubsection{LQC with multiplicative noise}

Let the dynamics be given by 
\begin{eqnarray}
dX &=& (MX + N\nu) dt + \sum_{i=1}^{n} C_i X \circ dB^i_t \\
&=& (\tilde M X + N \nu) dt + \sum_{i=1}^{n} C_i X dB^i_t.
\end{eqnarray}

Denote by $X^{t,x,\mu,{\mathbf{\eta}}}$ the solution starting from $x$ at
time $t$ to 
\begin{align*}
d X^{t,x,\mu,{\mathbf{\eta}}}_s = (MX^{t,x,\mu,{\mathbf{\eta}}}_s + N \mu)
dt + \sum_{i=1}^{n} C_i X^{t,x,\mu,{\mathbf{\eta}}}_s d{\mathbf{\eta}}^i_t \\
\end{align*}
and by $\Gamma_{t,s}$ the (matrix) solution to the RDE 
\begin{align*}
d_s \Gamma_{t,s} = M\Gamma_{t,s} ds + \sum_{i=1}^{n} C_i \Gamma_{t,s} d{%
\mathbf{\eta}}_s, \;\;\; \Gamma_{t,t} = I
\end{align*}
Then 
\begin{align*}
X^{t,x,\mu,{\mathbf{\eta}}}_s = \Gamma_{t,s} x + \int_t^s \Gamma_{r,s} N
\mu_r dr.
\end{align*}
For simplicity we now take $d=n=1$: the general case is only notationally
more involved.

The optimization problem is given by 
\begin{eqnarray}
V(t,x) &= \sup_{\nu \in \mathcal{A}} {\mathbb{E}} \left[ \frac{1}{2}
\int_t^T (Q X_s^2 + R \nu_s^2) ds + \frac{1}{2} G X_T^2 \right].
\end{eqnarray}

By Section 6.6 in \cite{YZ} the value function is again given as 
\begin{align*}
V(t,x) &= \frac{1}{2}P_t x^2
\end{align*}
and the optimal control as 
\begin{align*}
u^*(t,x) = -R^{-1} N P_t x,
\end{align*}
where 
\begin{align}  \label{eq:riccati}
\begin{split}
&\dot P_t + 2 P_t M + 2 P_t C^2 + Q - N^2 R^{-1}P_t^2 = 0, \qquad P_T = G.
\end{split}%
\end{align}
We can then compute explicitely the Davis--Burstein and Rogers penalties :

\begin{proposition}
For $t \leq r \leq T$, define 
\begin{align*}
\Theta_{r} &= \int_r^T P_s \Gamma^2_{r,s} (d{\mathbf{\eta}}_s - C ds).
\end{align*}

Then the optimal penalty corresponding to Theorem \ref{thm:dbDuality} is
given by 
\begin{align*}
z^{1}({\mathbf{\eta}},\mu) & = - CN x \int_t^T \Theta_s \mu_s ds,
\end{align*}
while the optimal penalty corresponding to Theorem \ref{thm:dualityUsingVF}
is given by 
\begin{align*}
z^{2}({\mathbf{\eta}},\mu) & = C \Theta_{t} x^2 + C N x \int_t^T \bar
\Gamma_{t,s} \Theta_{s} \mu_s ds + C N^2 \int_t^T \int_t^T \Gamma_{r \wedge
s, r\vee s} \Theta_{r \vee s} \mu_r \mu_s dr ds.
\end{align*}
\end{proposition}

\begin{proof}
The optimal penalty stemming from Theorem \ref{thm:dbDuality} (see also
Remark \ref{rem:DBrunningcost}) is given by $\int_t^T \lambda^*(r,x) \mu_r
dr $, where 
\begin{align*}
\lambda^*(r,x) &= N \left(G Z^1_T \partial_x Z^1_T + \partial_x Z^2_T\right)
- NP(r) x,
\end{align*}
where 
\begin{align*}
dZ^1 &= M Z^1 ds + C Z^1 d{\mathbf{\eta}}_t, \;\;\; Z^1_r=x; \\
dZ^2 &= \frac{1}{2} \left(Q-N^2R^{-1} P(s)^2 \right) (Z^1)^2 ds,\;\;\;
Z^2_r=0.
\end{align*}

Since $Z^1_s = \Gamma_{r,s} x$, this is computed to 
\begin{align*}
\lambda^*(r,x) &= N x \left( G \Gamma_{r,T}^2 + \int_r^T \left(Q-N^2R^{-1}
P(s)^2 \right) \Gamma_{r,s}^2 ds - P(r) \right). \\
&= Nx \left( \left[ P(s) \Gamma_{r,s}^2 \right]_{s=r}^T + \int_r^T \left( -
\dot P(s) - 2M P(s) - 2 C^2 P(s) \right) \Gamma_{r,s}^2 ds \right) \\
&= Nx \left( \int_r^T P(s) \Gamma_{r,s}^2 (C d{\mathbf{\eta}}_s - C^2 ds )
\right) \\
&= NCx \Theta_r.
\end{align*}

For the optimal penalty corresponding to Theorem \ref{thm:dualityUsingVF},
we apply again Remark \ref{rem:RogersInt} to see that the optimal penalty is
given by 
\begin{align*}
M_{t,T}^{t,x,\mu,{\mathbf{\eta}},V} &= \int_t^T \langle DV(s,X_s), C X d{%
\mathbf{\eta}}_s \rangle - \int_t^T \func{Tr}[ C^2 X^2 D^2 V(s,X_s) ]ds \\
&= \int_t^T P_r C |X^{t,x,\mu,{\mathbf{\eta}}}_r|^2 d{\mathbf{\eta}}_r -
\int_t^T C^2 |X^{t,x,\mu,{\mathbf{\eta}}}_r|^2 P_r dr.
\end{align*}

It only remains to perform straightforward computations expanding the
quadratic terms and applying Fubini's theorem.
\end{proof}


\section{Appendix: RDEs with controlled drift}

\begin{theorem}[RDE with controlled drift]
\label{Thm:RDEWithControlledDrift} Let $\alpha \in (1/3,1/2]$. 
Let ${\mathbf{\eta}} \in \mathcal{C}^{0,\alpha}$ a geometric $\alpha$%
-H\"older rough path. 
Let $\gamma > \frac{1}{\alpha}$. 
Let $U$ be the subset of a separable Banach space. Let $b: {\mathbb{R}} ^e
\times U \to {\mathbb{R}} ^e$ sucht that $b(\cdot,u) \in \func{Lip}^1( {%
\mathbb{R}} ^e)$ uniformly in $u \in U$ (i.e. $\sup_{u\in U} ||b(\cdot,u)||_{%
\func{Lip}^1( {\mathbb{R}} ^e)} < \infty$) and such that $u \mapsto
b(\cdot,u)$ is measurable. Let $\sigma_1,\dots,\sigma_d \in \func{Lip}%
^\gamma( {\mathbb{R}} ^e)$. Let $\mu: [0,T] \to U$ be measurable, i.e. $\mu
\in \mathcal{M}$.

(i) Then there exists a unique $Y \in \mathcal{C}^{0,\alpha}$ that solves 
\begin{equation*}
Y_t = y_0 + \int_0^t b( Y_r, \mu_r ) dr + \int_0^t \sigma( Y ) d{\mathbf{\eta%
}}_r.
\end{equation*}

Moreover the mapping 
\begin{align*}
(x_0,\mathbf{\eta}) \mapsto Y \in \mathcal{C}^{0,\alpha}
\end{align*}
is locally Lipschitz continuous, uniformly in $\mu\in\mathcal{M}$.


(ii) Assume moreoever that $U \ni u \mapsto b(\cdot,u) \in \func{Lip}^1( {%
\mathbb{R}} ^e)$ is Lipschitz. If we use the topology of convergence in
measure on $\mathcal{M}$ then 
\begin{align}  \label{eq:jointContinuity}
\begin{split}
\mathcal{M} \times {\mathbb{R}} ^e \times \mathcal{C}^{0,\alpha} &\to 
\mathcal{C}^{0,\alpha} \\
(\mu,x_0,{\mathbf{\eta}}) &\mapsto Y,
\end{split}%
\end{align}
is continuous.

(iii) If $\nu: \Omega \times [0,T] \to U$ is progressively measurable and $%
\mathbf{B}$ is the Stratonovich rough path lift of a Brownian motion $B$,
then 
\begin{align}  \label{eq:roughEqualsClassical}
\left. Y \right\vert_{\mu=\nu,{\mathbf{\eta}}=\mathbf{B}} = \tilde Y, \qquad 
\P -a.s.,
\end{align}
where $\tilde Y$ is the (classical) solution to the controlled SDE 
\begin{align*}
\tilde Y_t = y_0 + \int_0^t b( \tilde Y_r, \nu_r ) dr + \int_0^t \sigma(
\tilde Y ) \circ dB_r.
\end{align*}

(iv) If moreover $\sigma_1,\dots,\sigma_d \in \func{Lip}^{\gamma+2}( {%
\mathbb{R}} ^e)$, we can write $Y = \phi(t, \tilde Y_t)$ where $\phi$ is the
solution flow to the RDE 
\begin{align*}
\phi(t,x) = x + \int_0^t \sigma(\phi(r,x)) d{\mathbf{\eta}}_r,
\end{align*}
and $\tilde Y$ solves the classical ODE 
\begin{align*}
\tilde Y_t = x_0 + \int_0^t \tilde b(r, \tilde Y_r, \mu_r) dr,
\end{align*}
where we define componentwise 
\begin{align*}
\tilde b(t,x,u)_i = \sum_k \partial_{x_k} \phi^{-1}_i( t, \phi(t,x) )
b_k(\phi(t,x), u).
\end{align*}
\end{theorem}

\begin{remark}
In the last case, i.e. point (iv), we can immediately use results in \cite%
{FV} (Theorem 10.53) to also handle linear vector fields. 
\end{remark}

\begin{proof}
Denote for $\mu \in \mathcal{M}$ 
\begin{align*}
Z^\mu_t(\cdot) := \int_0^t b(\cdot,\mu_r) dr,
\end{align*}
which is a well defined Bochner integral in the space $\func{Lip}^1( {%
\mathbb{R}} ^e)$ (indeed, by assumption on $b$, $\int_0^t
||b(\cdot,\mu_r)||_{\func{Lip}^1( {\mathbb{R}} ^e)} dr < \infty$). Then $%
Z^\mu \in C^{1-H\ddot{o}lder}( [0,T], \func{Lip}^1( {\mathbb{R}} ^e) )$.
Indeed 
\begin{align*}
||Z^\mu_t - Z^\mu_s||_{\func{Lip}^1( {\mathbb{R}} ^e)} &= ||\int_s^t
b(\cdot, \mu_r) dr||_{\func{Lip}^1( {\mathbb{R}} ^e)} \\
&\le \int_s^t ||b(\cdot, \mu_r)||_{\func{Lip}^1( {\mathbb{R}} ^e)} dr \\
&\le |t-s| \sup_{u\in\mathcal{U}} ||b(\cdot, u)||_{\func{Lip}^1( {\mathbb{R}}
^e)},
\end{align*}
so 
\begin{align}  \label{eq:indpOfU}
||Z^\mu||_{1-H\ddot{o}lder} \le \sup_{u\in\mathcal{U}} ||b(\cdot, u)||_{%
\func{Lip}^1( {\mathbb{R}} ^e)} < \infty,
\end{align}
independent of $\mu \in \mathcal{M}$.

By Theorem \ref{thm:RDEExistenceLispchitzness} we get a unique solution to
the RDE 
\begin{align*}
dY = f(Y) dZ^\mu + \sigma(Y) d{\mathbf{\eta}}
\end{align*}
where $f: {\mathbb{R}} ^e \to L(\func{Lip}^1( {\mathbb{R}} ^e), {\mathbb{R}}
^e)$ is the evaluation operator, i.e. $f(y) V := V(y)$. 
This gives existence of the controlled RDE as well as continuity in the
starting point and in ${\mathbf{\eta}}$. By \eqref{eq:indpOfU}, this is
independent of $\mu \in \mathcal{M}$ and we hence have shown (i).

Concerning (ii), assume now that $U \ni u \mapsto b(\cdot,u) \in \func{Lip}%
^1 $ is Lipschitz. Using the representation given in the proof of (i) it is
sufficient to realize that if $\mu^n \to \mu \in \mathcal{M}$ in measure,
then $Z^{\mu^n} \to Z^\mu$ in $C^{ \beta-H\ddot{o}lder}([0,T],\func{Lip}^1( {%
\mathbb{R}} ^e))$, for all $\beta < 1$. 

Concerning (iii): first of all, 
we can regard $\nu$ as a measurable mapping from $(\Omega, {\mathcal{F}} )$
into the space of all measurable mappings from $[0,T] \to U$ with the
topology of convergence in measure. Indeed, if $U$ is a compact subset of a
separable Banach space, then this follows from the equivalence of weak and
strong measurability for Banach space valued mappings (Pettis Theorem, see
Section V.4 in \cite{yosida}). If $U$ is a general subset of a separable
Banach space, then define $\nu^n: \Omega \to \mathcal{M}$ with $%
\nu^n(\omega)_t := \Phi^n( \nu(\omega)_t )$. Here $\Phi^n$ is a (measurable)
nearest-neighbor projection on $\{x_1,\ldots, x_n\}$, the sequence $(x_k)_{k
\geq 0}$ being dense in the Banach space. Then $\nu^n$ is taking values in a
compact set and hence by the previous case, is measurable as a mapping to $%
\mathcal{M}$. Finally $\nu$ is the pointwise limit of the $\nu^n$ and hence
also measurable. 


Hence $\left. Y \right\vert_{\mu=\nu,{\mathbf{\eta}}=\mathbf{B}}$ is
measurable, as the concatentation of measurable maps (here we use the joint
continuity of RDE solutions in the control and the rough path, i.e.
continuity of the mapping \eqref{eq:jointContinuity}).

Now, to get the equality \eqref{eq:roughEqualsClassical}: we can argue as in 
\cite{FH} using the Riemann sum representation of stochastic integral. 

(iv) This follows from Theorem 1 in \cite{D12} or Theorem 2 in \cite{CDFO}.
\end{proof}

\begin{remark}
One can also prove ``by hand'' existence of a solution, using a fixpoint
argument, like the one used in \cite{gubinelli}. This way one arrives at the
same regularity demands on the coefficients. Using the infinite dimensional
setting makes it possible to immediately quote existing results on
existence, which shortens the proof immensely. We thank Terry Lyons for
drawing our attention to this fact. 
\end{remark}

In the proof of the previous theorem we needed the following version of
Theorem 6.2.1 in \cite{LQ}.

\begin{theorem}
\label{thm:RDEExistenceLispchitzness} Let $V,W,Z$ be some Banach spaces. Let
tensor products be endowed with the projective tensor norm. \footnote{%
See \cite{LyonsStFlour} p. 18 for more on the choice of tensor norms which,
of course, only matter in an infinite dimensional setting.} Let $\alpha \in
(1/3,1/2]$. and ${\mathbf{\eta}} \in \mathcal{C}^{0,\alpha}(W)$ and $Z \in
C^{\beta-H\ddot{o}lder}([0,T], V)$ for some $\beta > 1 - \alpha$. Let $f: Z
\to L(V,Z)$ be $\func{Lip}^1$, let $g: Z \to L(W,Z)$ be $\func{Lip}^\gamma$, 
$\gamma > p$. Then there exists a unique solution $Y \in \mathcal{C}%
^{0,\alpha}(Z)$ to the RDE 
\begin{align*}
dY = f(Y) dZ + g(Y) d{\mathbf{\eta}},
\end{align*}
in the sense of Lyons. \footnote{%
See e.g. Definition 5.1 in \cite{LyonsStFlour}.}

Moreover for every $R>0$ there exists $C = C(R)$ such that 
\begin{align*}
\rho_{\alpha-\text{H\"{o}l}}( Y, \bar Y ) \le C ||Z - \bar Z||_{\beta-H\ddot{%
o}lder}
\end{align*}
whenever $(Z,X)$ and $(\bar Z, X)$ are two driving paths with $||Z||_{\beta-H%
\ddot{o}lder;[0,T]}, ||\bar Z||_{\beta-H\ddot{o}lder;[0,T]}, ||X||_{\alpha-H%
\ddot{o}lder;[0,T]} \le R$.
\end{theorem}

\begin{proof}
Since $Z$ and $X$ have complementary Young regularity (i.e. $\alpha + \beta
> 1$) the joint rough path $\lambda$ over $(Z, X)$, where the missing
integrals of $Z$ and the cross-integrals of $Z$ and $X$ are defined via
Young integration. So we have 
\begin{align*}
\lambda_{s,t} = 1 + \left( 
\begin{array}{c}
Z_{s,t} \\ 
X_{s,t}%
\end{array}
\right) + \left( 
\begin{array}{cc}
\int_s^t Z_{s,r} \otimes dZ_r & \int_s^t Z_{s,r} \otimes d\eta_r \\ 
\int_s^t \eta_{s,r} \otimes dZ_r & \int_s^t \eta_{s,r} \otimes d\eta_r%
\end{array}
\right)
\end{align*}

Then, by Theorem 6.2.1 in \cite{LQ}, there exists a unique solution to the
RDE 
\begin{align*}
dY = h(Y) d\lambda,
\end{align*}
where $h = (f, g)$.

We calculate how $\lambda$ depends on $Z$. For the first level we have of
course $||\lambda^{(1)} - \bar \lambda^{(1)}||_{\alpha-H\ddot{o}lder} \le
||Z - \bar Z||_{\beta-H\ddot{o}lder}$. 
For the second level we have, by Young's inequality, 
\begin{align*}
\sup_{s < t} \frac{|\int_{s}^{t} Z_{s,r} dZ_r - \int_{s}^{t} \bar Z_{s,r}
d\bar Z_r|}{|t-s|^{2\alpha}} &\le \sup_{s < t} \frac{|\int_{s}^{t} Z_{s,r} d%
\left[ Z_r - \bar Z_r \right]}{|t-s|^{2\alpha}} + \sup_{s < t} \frac{%
|\int_{s}^{t} Z_{s,r} - \bar Z_{s,r} d\bar Z_r|}{|t-s|^{2\alpha}} \\
&\le C ||Z||_{\beta-H\ddot{o}lder} ||Z-\bar Z||_{\beta-H\ddot{o}lder},
\end{align*}
and similarily 
\begin{align*}
\sup_{s < t} \frac{|\int_{s}^{t} X_{s,r} dZ_r - \int_{s}^{t} X_{s,r} d\bar
Z_r|}{|t-s|^{2\alpha}} \le C ||Z - \bar Z||_{\beta-H\ddot{o}lder;[s,t]}.
\end{align*}


Plugging this into the continuity estimate of Theorem 6.2.1 in \cite{LQ} we
get 
\begin{align*}
\rho_{\alpha-\text{H\"{o}l}}( Y, \bar Y ) \le C ||Z - \bar Z||_{\beta-H\ddot{%
o}lder}
\end{align*}
as desired.
\end{proof}

\end{document}